\documentclass[12pt,a4paper,reqno]{amsart}
\usepackage{amsmath,amssymb,amsthm}
\usepackage{tikz-cd}
\usepackage[margin=1.15in]{geometry}
\usepackage{hyperref}

\usepackage{cleveref}
\crefformat{equation}{(#2#1#3)}
\Crefformat{equation}{(#2#1#3)}

\newtheorem{theorem}{Theorem}
\newtheorem{maintheorem}[theorem]{Theorem}
\newtheorem{proposition}{Proposition}
\newtheorem{lemma}{Lemma}
\newtheorem{corollary}{Corollary}

\theoremstyle{remark}
\newtheorem{remark}{Remark}
\newtheorem{example}{Example}

\crefname{maintheorem}{Theorem}{Theorems}
\Crefname{maintheorem}{Theorem}{Theorems}

\newcommand{\fe}{\mathfrak{e}}
\newcommand{\fo}{\mathfrak{o}}
\newcommand{\fa}{\mathfrak{a}}
\newcommand{\fb}{\mathfrak{b}}
\newcommand{\C}{\mathbb{C}}
\newcommand{\D}{\mathbb{D}}
\newcommand{\R}{\mathbb{R}}
\newcommand{\Lom}{\mathfrak{L}}
\newcommand{\sjz}{j_{\nu,1}}

\title{Lommel polynomials and explicitly solvable prediction problems on the unit circle}

\author{Steven P. Clark}
\address{University of North Carolina at Charlotte}
\email{spclark@charlotte.edu}


\begin{document}

\begin{abstract}
To each finite symmetric measure $\sigma$ on the real line, with support a compact subset of $(-2,2)$,
we associate a measure $\mu$ on the unit circle by transporting mass at $\pm x$ to $e^{\pm i\theta(x)}$, $\theta(x)=2\arcsin(x/2)$.
The linear prediction errors, Verblunsky coefficients, and Toeplitz determinants of $\mu$ are then expressed through the orthogonal polynomial data of $\sigma$ at the single edge point $x=2$.
In particular $E_m(\mu)=\tfrac12 t_m\|P_m\|_\sigma^2$ with $t_m=P_{m+1}(2)/P_m(2)$.
Under a condition on the first coefficients, decay of the recurrence coefficients of $\sigma$ forces the $t_m$ to increase from $t_1$ onward,
a Tur\'an-type monotonicity in the degree placing every prediction margin of $\mu$ past the first above the corresponding coefficient of $\sigma$.
Taking $\sigma$ to be the Lommel-polynomial measures, with atoms at rescaled reciprocals of the zeros of the Bessel function $J_\nu$,
yields a two-parameter family of purely atomic circle measures with a normalized determinant limit equal to a value of $J_\nu$.
\end{abstract}

\maketitle

\section{Introduction}\label{sec:intro}

Let $\D$ denote the open unit disk in $\C$ and let $\mu$ be a finite positive measure with infinite support on the unit circle $\partial\D$.
Writing $\mathcal{P}_{m-1}$ for the polynomials in $\C[z]$ of degree at most $m-1$, with $\mathcal{P}_{-1}=\{0\}$, let
\begin{equation}\label{eq:monicprederrors}
E_m(\mu)=\min\Bigl\{\int_{\partial\D}|\Phi|^2\,d\mu:\ \Phi\in z^m+\mathcal{P}_{m-1}\Bigr\},\qquad m\ge0,
\end{equation}
be the \emph{monic prediction errors} of $\mu$.
The minimizers $\Phi_m$ are the monic orthogonal polynomials of $\mu$, the values $\Phi_m(0)$ carry the \emph{Verblunsky coefficients},
and the ratios $E_m/E_{m-1}=1-|\Phi_m(0)|^2$ are the \emph{prediction margins}.\footnote{
Our OPUC definition differs slightly from that of a standard modern reference.
Unlike Simon \cite{Simon1,Simon2}, we do not normalize $\mu$ to a probability measure, since the total mass plays a role in what follows.}
These quantities are central to the Szeg\H{o} theory of orthogonal polynomials on the unit circle (OPUC)
\cite{Szego,Simon1} and to linear prediction \cite{GrenanderSzego,Simon1}.

Families for which the $E_m$ admit closed forms are rare. The standard examples, the Bernstein-Szeg\H{o} measures, are absolutely continuous,
and explicit purely atomic examples are scarcer still.

In this paper we exhibit a large class of solvable purely atomic cases.
The construction starts from a finite symmetric measure $\sigma$ on the real line with $\operatorname{supp}\sigma\subset(-2,2)$, and
pushes it forward to the circle under $x\mapsto e^{i\theta(x)}$, where
\begin{equation}\label{eq:arcsinmap}
\theta(x)=2\arcsin\bigl(x/2\bigr),\qquad -2<x<2.
\end{equation}
Since $\theta$ is an odd function, the pair $\pm x$ maps to the conjugate pair $e^{\pm i\theta(x)}$,
and we call the resulting circle measure $\mu=\Lambda\sigma$ the \emph{arcsin lift} of $\sigma$.

This particular lift is chosen so that the image of $\mu$ under $z=e^{i\theta}\longmapsto s=|z-1|^2=2-2\cos\theta$
coincides with the image of $\sigma$ under $x\mapsto s=x^2$. Indeed, $|e^{i\theta(x)}-1|^2=2-2\cos\theta(x)=x^2$,
so that the chord length from $e^{i\theta(x)}$ to $1$ on the unit circle is $|x|$.

Our first result expresses the prediction theory of $\mu$ through a single scalar sequence built from the
orthogonal polynomials of $\sigma$ evaluated at the edge $x=2$, the endpoint corresponding to the antipode $z=-1$.

\begin{maintheorem}\label{thm:main}
Let $\sigma$ be a finite positive symmetric measure with infinite support contained in $[-R,R]$, $R<2$,
and let $P_m$ be its monic orthogonal polynomials satisfying the recursion
\[P_{m+1}(x)=x\,P_m(x)-A_m^2\,P_{m-1}(x),\quad m\geq 1,\]
with $P_0=1$ and $P_1=x$. For $m\ge0$, define the edge ratio, $t_m$, by
\begin{equation}\label{eq:tdef}
t_m=\frac{P_{m+1}(2)}{P_m(2)},\qquad\text{so that}\qquad t_0=2,\quad t_m=2-\frac{A_m^2}{t_{m-1}} .
\end{equation}
Then for the arcsin lift $\mu=\Lambda\sigma$, whose monic orthogonal polynomials $\Phi_m$ have real
coefficients by conjugation symmetry, so that $\Phi_m(0)\in\R$ (see the opening of the proof of \Cref{lem:split}),
\begin{equation}\label{eq:mainE}
E_m(\mu)=\frac{t_m}{2}\,\|P_m\|_\sigma^2, \qquad \Phi_m(0)=(-1)^m\,(t_m-1), \quad\text{for}\quad m\ge0.
\end{equation}
Consequently, for $m\ge1$ and $N\ge0$,
\begin{eqnarray}
\frac{E_m}{E_{m-1}}&=&1-\Phi_m(0)^2=t_m\,(2-t_m)=A_m^2\,\frac{t_m}{t_{m-1}},\label{eq:mainmargin}\\
D_N(\mu)&=&\frac{P_{N+1}(2)}{2^{N+1}}\,\prod_{m=0}^{N}\|P_m\|_\sigma^2,\label{eq:maindet}
\end{eqnarray}
where $D_N=\det\bigl[c_{j-k}\bigr]_{j,k=0}^{N}$, $c_k=\int_{\partial\D}z^{-k}\,d\mu$, is the Toeplitz determinant of order $N+1$ of $\mu$.
\end{maintheorem}

\Cref{lem:bijection} establishes that every nontrivial conjugation-symmetric measure on $\partial\D$ supported away from $z=-1$ is the arcsin lift of a unique such $\sigma$.
Thus \Cref{thm:main} is a statement about every conjugation-symmetric measure with infinite support on a proper closed sub-arc of $\partial\D$ avoiding $z=-1$.
The proof goes through the even and odd reciprocal splittings of the extremal polynomials;
the four constrained minima that appear all collapse,
two through the even and odd reductions and two through one Christoffel transform evaluated at the antipodal point,
to values of the $P_m$ at $x=2$.

These constructs have deep roots.
The splitting of $\Phi_m$ into reciprocal and anti-reciprocal parts underlies the split Levinson algorithm of
Delsarte and Genin \cite{DelsarteGenin86}, and the correspondence between measures on the circle with
real Verblunsky coefficients and symmetric measures on an interval through $2x=z^{1/2}+z^{-1/2}$, a
rotation of \Cref{eq:arcsinmap}, is their transformation \cite{DelsarteGenin91}; see also the classical
Szeg\H{o} mapping and the Geronimus relations \cite{Szego,Geronimus,Simon2}.
In the chain sequence form of Costa, Felix and Sri Ranga \cite{CostaFelixRanga} and continued in
\cite{CastilloCostaRangaVeronese,BraccialiSilvaRangaVeronese}, every nontrivial probability measure on the
circle corresponds to a pair of real sequences, one a positive chain sequence, and for the
conjugation-symmetric measures considered here the Verblunsky coefficients are $\pm(1-2g_m)$ with
$\{g_m\}$ the minimal parameter sequence of that chain sequence in the sense of Wall \cite{Wall} and Chihara \cite{Chihara}.
The value formula in \Cref{eq:mainE} is equivalent, after the rotation $z\mapsto-z$, to this correspondence,
with parameters $g_m=1-t_m/2$, so that $\{A_m^2/4\}$ is the chain sequence and the recurrence coefficients satisfy $A_m^2=4g_m(1-g_{m-1})$,
in the normalization with support in $[-2,2]$.
By Favard's theorem \cite{Chihara} and Wall's criterion every sequence with $g_0=0$ and $g_m\in(0,1)$ arises from a symmetric measure on $[-2,2]$ with infinite support,
and \Cref{thm:main} parametrizes those measures supported strictly inside $(-2,2)$ by their edge data.

What \Cref{thm:main} contributes to this lineage is the explicit transport to the line measure $\sigma$,
the chain parameters realized as the edge ratios \Cref{eq:tdef}, whose sequence we call the \emph{edge chain} of $\sigma$,
and the closed factorizations of the prediction errors and Toeplitz determinants through the norms of $\sigma$ and the single value $P_{N+1}(2)$,
which we have not previously seen in this form.

We find that, under a condition on the first recurrence coefficients, the recursion in \Cref{eq:tdef} converts decay of the recurrence coefficients
into growth of the ratios $t_m$ from $t_1$ onward, a Tur\'an-type log-convexity of $m\mapsto P_m(2)$ for $m\ge2$ at the edge $x=2$, beyond the spectrum of $\sigma$.
By \Cref{thm:mono}, this forces every prediction margin of the lift past the first to exceed
the corresponding $A_m^2$ by the monotone factor $t_m/t_{m-1}$.
Tur\'an-type inequalities for Lommel functions and polynomials, Bessel functions, and related symmetric orthogonal polynomial systems
have a substantial literature \cite{BariczTuran,BustozIsmail} on the interval of orthogonality,
where the Tur\'an determinant is nonnegative.
Beyond the spectrum it turns negative, classically for the ultraspherical family \cite{VenkatachaliengarRao},
and the log-convexity here is that sign change for the Lommel-polynomial sequence.

We are then able to exhibit a family for which the construction is explicitly solvable.
We call a purely atomic symmetric measure \emph{canonical} when the atom at each $\pm x$ has mass $x^2$
(the lift then weights each atom by the squared chord $|1-z|^2$ to the point $z=1$).
For $\nu>-1$
let $0<j_{\nu,1}<j_{\nu,2}<\cdots$ be the positive zeros of the Bessel function of the first kind $J_\nu$.
For $0<\beta<2\sjz$ let $\Lom_{\nu,\beta}$ be the purely atomic measure placing mass $(\beta/j_{\nu,k})^2$ at each of the points $\pm\beta/j_{\nu,k}$.
The orthogonal polynomials of $\Lom_{\nu,\beta}$ are dilated Lommel polynomials \cite{Dickinson,DickinsonPollakWannier,Watson,Chihara,Ismail,Koelink}, with
\begin{equation}\label{eq:lommelA}
A_m^2\bigl(\Lom_{\nu,\beta}\bigr)=\frac{\beta^2}{4(\nu+m)(\nu+m+1)},\qquad m\ge1.
\end{equation}
We are not aware of the Lommel polynomials, or the measures on $\partial\D$ they generate under the lift, appearing previously in the OPUC literature.

\Cref{thm:lommel} shows that the lifts $\Lambda\Lom_{\nu,\beta}$ form a two-parameter family of
purely atomic measures on $\partial\D$ for which the central quantities of the Szeg\H{o} theory are given explicitly.
The Verblunsky coefficients are $\alpha_{m-1}=(-1)^{m+1}(t_m-1)$, the Toeplitz determinants are Gamma-function products
times a single Lommel polynomial value, and as $N\to\infty$ the normalized edge product $P_{N+1}(2)/2^{N+1}$
converges, by Hurwitz's theorem, to the Bessel value $\Gamma(\nu+1)(4/\beta)^\nu J_\nu(\beta/2)$;
the evaluation point is the antipode $z=-1$, off the spectrum.
The deviations obey a third-order law,
\begin{equation*}
\log\frac{1-\Phi_m(0)^2}{A_m^2} = \log\frac{t_m}{t_{m-1}}
= \frac{\beta^2}{8(\nu+m-1)(\nu+m)(\nu+m+1)}\,\bigl(1+O(m^{-2})\bigr),
\end{equation*}
positive for every $m\ge2$ once the first recurrence coefficients satisfy the condition of \Cref{thm:mono}.
The Bessel function appears through the support and through this normalized limit, each time in closed form,
while the deviation law is rational.

The two real-line sequences behind \Cref{thm:main}, the edge values $P_m(2)$ and the norms $\|P_m\|_\sigma^2$,
combine into the Christoffel function of $\sigma$ at $x=2$.
\Cref{sec:christoffel} casts the determinant formula and Bessel limit as statements about this function,
identifies $\mathcal J_{\nu,\beta}$ as the limiting ratio of the circle Toeplitz determinant to the line Hankel determinant,
and gives closed-form asymptotics for the Christoffel function of the Lommel ensemble,
its super-exponential decay at the antipode $z=-1$ recording the support's accumulation at $z=1$.
This is why we package the theory through the value at $x=2$ rather than the Verblunsky coefficients,
since the degree-monotonicity and the Lommel collapse are both assertions about that single point,
which the edge data displays and the Verblunsky coefficients do not.

\section{The arcsin lift}\label{sec:lift}
Throughout, $\sigma$ is a finite positive symmetric measure on $\R$ with infinite support contained in
$[-R,R]$, $R<2$, and $\tau_0=\sigma(\R)$. Its monic orthogonal polynomials $P_m$ satisfy the symmetric recurrence
\[P_{m+1}(x)=x\,P_m(x)-A_m^2\,P_{m-1}(x),\qquad A_m^2>0,\]
the diagonal coefficients $\langle x P_m,P_m\rangle/\|P_m\|_\sigma^2$ vanishing because $\sigma$ is symmetric,
and $\|P_m\|_\sigma^2=\tau_0\prod_{j=1}^{m}A_j^2$, with the empty product at $m=0$ being $1$ as usual.

Let $\mathcal{M}_{(-2,2)}^{sym}$ denote the set of all finite positive symmetric measures $\sigma$ on $\R$ with infinite support and $\operatorname{supp}(\sigma)\subset(-2,2)$.
For the unit circle measures, let $\mathcal{M}_{\partial\D\setminus\{-1\}}^{\overline{sym}}$ denote the set of all
finite positive conjugation symmetric measures $\mu$ on $\partial\D$ with infinite support and $-1\notin\operatorname{supp}(\mu)$.

Since the support of a finite measure on $\R$ is closed in $\R$, the condition $\operatorname{supp}(\sigma)\subset(-2,2)$
already makes $\operatorname{supp}(\sigma)$ a compact subset of $(-2,2)$, hence contained in some $[-R,R]$ with $R<2$,
so the two phrasings used in this paper are equivalent.

The arcsin lift $\mu=\Lambda\sigma$ is the pushforward $\Psi_*\sigma$ of $\sigma$ under the transport map
$\Psi(x)=e^{\,i\,\operatorname{sgn}(x)\theta(|x|)}$, with $\theta$ as in \Cref{eq:arcsinmap} and $\Psi(0)=1$,
the measure on $\partial\D$ determined by $\mu(B)=\sigma\bigl(\Psi^{-1}(B)\bigr)$ for every Borel set $B$.
The map $\Psi\colon(-2,2)\to \partial\D\setminus\{-1\}$, $\Psi(x)=e^{i\theta(x)}$ with $\theta(x)=2\arcsin(x/2)$,
is a homeomorphism, being the composition of the continuous increasing bijection $x\mapsto\theta(x)$ of $(-2,2)$ onto $(-\pi,\pi)$
with $\theta\mapsto e^{i\theta}$; its inverse is $\Psi^{-1}(e^{i\theta})=2\sin(\theta/2)$, $|\theta|<\pi$.
In particular $\Psi$ is a Borel isomorphism, carrying Borel sets to Borel sets with $\Psi^{-1}(\Psi(A))=A$ for every Borel $A\subset(-2,2)$.

The following lemma is routine measure theory, and a reader willing to trust it may pass directly to \Cref{sec:split}.
We include a detailed proof because several separate facts require checking,
injectivity through the Borel isomorphism, conjugation symmetry on both sides through the intertwining \Cref{eq:conjintertwine},
and membership and surjectivity through supports, and everything after this section stands on them.

\begin{lemma}\label{lem:bijection}
The arcsin lift $\Lambda$ is a bijection from $\mathcal{M}_{(-2,2)}^{sym}$
onto $\mathcal{M}_{\partial\D\setminus\{-1\}}^{\overline{sym}}$.
It preserves total mass, and its inverse is the pushforward under $e^{i\theta}\mapsto 2\sin(\theta/2)$, $|\theta|<\pi$.
\end{lemma}

\begin{proof}
Let $\sigma_1,\sigma_2\in\mathcal{M}_{(-2,2)}^{sym}$ such that $\Lambda\sigma_1=\Lambda\sigma_2$.
By definition $(\Lambda\sigma_i)(B)=\sigma_i\bigl(\Psi^{-1}(B)\bigr)$ for Borel $B\subset\partial\D\setminus\{-1\}$.
Because $\operatorname{supp}(\sigma_i)\subset(-2,2)$, each $\sigma_i$ vanishes off $(-2,2)$ and is determined by its values on Borel subsets of $(-2,2)$.
For such a set $A$, the image $\Psi(A)$ is Borel and
\[\sigma_i(A)=\sigma_i\bigl(\Psi^{-1}(\Psi(A))\bigr)=(\Psi_*\sigma_i)(\Psi(A))=(\Lambda\sigma_i)(\Psi(A)).\]
The right-hand side is independent of $i$ by hypothesis, so $\sigma_1(A)=\sigma_2(A)$ for every Borel $A\subset(-2,2)$,
and therefore $\sigma_1=\sigma_2$. Thus $\Lambda$ is injective.

Since $\arcsin$ is odd,
\begin{equation}\label{eq:conjintertwine}
\Psi(-x)=e^{-2i\arcsin(x/2)}=\overline{\Psi(x)},\qquad x\in(-2,2).
\end{equation}

Next, $\Lambda$ maps $\mathcal{M}_{(-2,2)}^{sym}$ into $\mathcal{M}_{\partial\D\setminus\{-1\}}^{\overline{sym}}$.
Let $\sigma\in\mathcal{M}_{(-2,2)}^{sym}$, so that $(\Lambda\sigma)(\partial\D)=\sigma\bigl(\Psi^{-1}(\partial\D)\bigr)=\sigma(\R)$.
The set $\Psi(\operatorname{supp}(\sigma))$ is an infinite compact subset of $\partial\D\setminus\{-1\}$,
$\Psi$ being injective and continuous on the compact set $\operatorname{supp}(\sigma)$,
and its complement in $\partial\D$ is open with $(\Lambda\sigma)$-mass $\sigma\bigl((-2,2)\setminus\operatorname{supp}(\sigma)\bigr)=0$.
Since $\Psi$ carries neighborhoods in $(-2,2)$ to neighborhoods in $\partial\D\setminus\{-1\}$,
every point of $\Psi(\operatorname{supp}(\sigma))$ lies in $\operatorname{supp}(\Lambda\sigma)$,
so $\operatorname{supp}(\Lambda\sigma)=\Psi(\operatorname{supp}(\sigma))$ is infinite and omits $-1$.
For Borel $B\subset\partial\D\setminus\{-1\}$, \Cref{eq:conjintertwine} gives $\Psi^{-1}(\overline B)=-\Psi^{-1}(B)$,
so the symmetry of $\sigma$ yields
\[(\Lambda\sigma)(\overline B)=\sigma\bigl(-\Psi^{-1}(B)\bigr)=\sigma\bigl(\Psi^{-1}(B)\bigr)=(\Lambda\sigma)(B),\]
and $\Lambda\sigma$ is conjugation symmetric.

Let $\mu\in\mathcal{M}_{\partial\D\setminus\{-1\}}^{\overline{sym}}$.
As $\operatorname{supp}(\mu)$ is closed in the compact space $\partial\D$ and does not contain $z=-1$, it is a compact subset of $\partial\D\setminus\{-1\}$.
In particular $\mu(\{-1\})=0$, so $\mu$ is carried by $\partial\D\setminus\{-1\}$ and may be regarded as a finite Borel measure on $\partial\D\setminus\{-1\}$.
Define $\sigma=(\Psi^{-1})_*\mu$, a finite positive Borel measure on $(-2,2)$, extended to a Borel measure on $\R$ by defining it to be zero on $\R\setminus (-2,2)$.
Its total mass is $\sigma(\R)=\mu(\partial\D)$, and for every Borel $A\subset(-2,2)$,
\begin{equation}\label{eq:sigmaformula}
\sigma(A)=\mu\bigl(\Psi(A)\bigr),
\end{equation}
since $(\Psi^{-1})^{-1}(A)=\Psi(A)$.
Because $\Psi^{-1}$ is a homeomorphism, $\operatorname{supp}(\sigma)=\Psi^{-1}(\operatorname{supp}(\mu))$.
To see this, realize that a point $x$ lies in $\operatorname{supp}(\sigma)$ if and only if every neighborhood has positive $\sigma$-mass,
and \eqref{eq:sigmaformula} turns neighborhoods of $x$ into neighborhoods of $\Psi(x)$ in $\partial\D\setminus\{-1\}$.
Hence $\operatorname{supp}(\sigma)$ is the continuous image of the compact set $\operatorname{supp}(\mu)$, so it is a compact subset of $(-2,2)$,
implying $\operatorname{supp}(\sigma)\subset(-2,2)$.
Moreover, $\operatorname{supp}(\sigma)$ is infinite because $\Psi^{-1}$ is injective and $\operatorname{supp}(\mu)$ is infinite.

For Borel $A\subset(-2,2)$, using \eqref{eq:sigmaformula}, then \eqref{eq:conjintertwine}, then the conjugation symmetry of $\mu$,
\[\sigma(-A)=\mu\bigl(\Psi(-A)\bigr)=\mu\bigl(\overline{\Psi(A)}\bigr)=\mu\bigl(\Psi(A)\bigr)=\sigma(A),\]
showing that $\sigma$ is symmetric.

Since pushforward satisfies $\Psi_*\circ\Psi^{-1}_*=(\Psi\circ\Psi^{-1})_*$ and $\Psi\circ\Psi^{-1}$ is the identity on $\partial\D\setminus\{-1\}$, we have
\[\Lambda\sigma=\Psi_*\sigma=\Psi_*(\Psi^{-1})_*\mu=(\Psi\circ\Psi^{-1})_*\mu=\mu\]
as measures on $\partial\D\setminus\{-1\}$. Since $\sigma((-2,2))=\mu(\partial\D\setminus\{-1\})=\mu(\partial\D)$,
the lift $\Lambda\sigma$ also gives zero mass to $\{-1\}$, so the equality holds on $\partial\D$.
Thus $\Lambda$ is surjective.
\end{proof}

The lift is supported in the arc $|\theta|\le\theta(R)<\pi$, with $\mu(\partial\D)=\tau_0$.
Since $2-2\cos\theta(x)=4\sin^2(\theta(x)/2)=x^2$, the image of $\mu$ under $e^{i\theta}\mapsto s=2-2\cos\theta$
equals the image $\omega$ of $\sigma$ under $x\mapsto s=x^2$; we write $\omega$ for this common measure on $[0,R^2]$ and $\eta=s\,d\omega$.
So we work with lifts without loss of generality by \Cref{lem:bijection}.

We frequently use the standard even and odd reductions of $\sigma$.
The maps $f(s)\mapsto f(x^2)$ and $f(s)\mapsto x f(x^2)$ are isometries of $L^2(\omega)$ and $L^2(\eta)$ onto the even and odd subspaces of
$L^2(\sigma)$, carrying monic polynomials of degree $n$ to those of degrees $2n$ and $2n+1$.
Hence the monic orthogonal polynomials $q_n$ of $\omega$ and $r_n$ of $\eta$ satisfy
\begin{equation}\label{eq:evenodd}
q_n(x^2)=P_{2n}(x),\qquad x\,r_n(x^2)=P_{2n+1}(x),\qquad \|q_n\|_\omega^2=\|P_{2n}\|_\sigma^2,\qquad \|r_n\|_\eta^2=\|P_{2n+1}\|_\sigma^2 .
\end{equation}
In particular, at $x=2$, that is $s=4$,
\begin{equation}\label{eq:edgevalues}
q_n(4)=P_{2n}(2),\qquad r_n(4)=\tfrac12\,P_{2n+1}(2).
\end{equation}
Since the zeros of these polynomials lie in the convex hull of the supports, $P_m(2)>0$, $q_n(4)>0$
and $r_n(4)>0$ for all degrees, and the edge chain \Cref{eq:tdef} is well defined and positive,
the recursion in \Cref{eq:tdef} is the three-term recurrence at $x=2$ together with $P_1(x)=x$.

\section{The splitting and the proof of \Cref{thm:main}}\label{sec:split}

The next lemma is a block-diagonalization. Write a polynomial $\Phi(z)=\sum_{k=0}^m c_kz^k$ in the coordinates $(c_0,\dots,c_m)$.
The reciprocal involution $\Phi\mapsto\Phi^*$ reverses their order, $(c_0,\dots,c_m)\mapsto(c_m,\dots,c_0)$,
so it acts as the flip matrix, with ones on the antidiagonal and zeros elsewhere.
Its $\pm1$ eigenspaces are the reciprocal and anti-reciprocal polynomials, and a real symmetric Toeplitz matrix commutes with the flip,
so the Gram matrix preserves these eigenspaces and splits into two blocks.
The vanishing cross term below is the statement that the off-diagonal block is zero, and the four minima are the half-line problems on the two blocks.

\begin{lemma}\label{lem:split}
Let $\mu=\Lambda\sigma$ be the arcsin lift, and write $d\eta=s\,d\omega$ for the odd reduction.
With every minimum taken over monic $Q\in\mathbb{R}[s]$, define for $n\ge1$
\[\fe_n=\min_{\deg Q=n}\int Q^2\,d\omega,\qquad \fo_n=\min_{\deg Q=n-1}\int s(4-s)\,Q^2\,d\omega,\]
and for $n\ge0$
\[\fa_n=\min_{\deg Q=n}\int s\,Q^2\,d\omega,\qquad \fb_n=\min_{\deg Q=n}\int (4-s)\,Q^2\,d\omega.\]
Then $E_0(\mu)=\tau_0$ and $\Phi_0(0)=1$, while
\[E_{2n}(\mu)=\frac{\fe_n\fo_n}{\fe_n+\fo_n},\quad n\ge1,\qquad E_{2n+1}(\mu)=\frac{\fa_n\fb_n}{\fa_n+\fb_n},\quad n\ge0,\]
and the extremal polynomials satisfy
\[\Phi_{2n}(0)=\frac{\fo_n-\fe_n}{\fo_n+\fe_n},\quad n\ge1, \qquad \Phi_{2n+1}(0)=\frac{\fa_n-\fb_n}{\fa_n+\fb_n},\quad n\ge0.\]
\end{lemma}

\begin{proof}
For $m=0$ the only monic competitor is $\Phi_0=1$, so $E_0(\mu)=\mu(\partial\D)=\tau_0$ and $\Phi_0(0)=1$.
Fix $m\ge1$. We first dispose of existence, uniqueness, and reality,
so that it will suffice to treat real monic $\Phi$.
Let $G$ be the Gram matrix of $1,z,\dots,z^{m-1}$ in $L^2(\mu)$.
Then $G_{jk}=\int z^{j-k}d\mu=c_{k-j}$ and $G$ is Toeplitz.
For any nonzero coefficient vector $a$ we have $a^*Ga=\|\sum \overline{a_j}\,z^j\|_\mu^2>0$,
since a nonzero polynomial of degree less than $m$ has at most $m-1$ zeros and cannot vanish $\mu$-a.e. on the infinite support of $\mu$.
Since $G$ is positive definite, $\int|\Phi|^2d\mu$ is a positive definite quadratic in the lower coefficients of $\Phi$,
hence coercive and strictly convex, and it attains a unique minimum over $\{z^m+\mathcal{P}_{m-1}\}$.
Since $\theta$ is odd and $\sigma$ is symmetric, $c_k=\int_{\partial\D}z^{-k}\,d\mu=\int_\R e^{-ik\theta(x)}\,d\sigma(x)$ is real,
and $c_{-k}=\overline{c_k}=c_k$ for every $k$ since $\mu$ is positive.
The objective $\int|\Phi|^2d\mu=\sum_{j,k}a_j\overline{a_k}\,c_{k-j}$ is therefore unchanged when every coefficient of $\Phi$ is conjugated,
so if $\Phi_m$ minimizes, so does $z\mapsto\overline{\Phi_m(\bar z)}$, and uniqueness identifies them.
It thus suffices to minimize over real monic $\Phi$ of degree $m$.

Set $\Phi^*(z)=z^m\Phi(1/z)$ and split $\Phi=R+I$ with $R=\tfrac12(\Phi+\Phi^*)$ and $I=\tfrac12(\Phi-\Phi^*)$,
so $R^*=R$ (reciprocal) and $I^*=-I$ (anti-reciprocal).
On $|z|=1$ we have $\bar R=z^{-m}R$ and $\bar I=-z^{-m}I$, hence
\[2\,\mathrm{Re}\,\bigl(R\bar I\bigr)=R\bar I+\bar R\,I=-z^{-m}RI+z^{-m}RI=0\]
pointwise on the circle, and the cancellation uses no property of $\mu$.
Thus
\[\int|\Phi|^2d\mu=\int|R|^2d\mu+\int|I|^2d\mu,\]
and for a fixed split of the leading coefficient the two pieces may be minimized independently.

Reciprocity forces zeros at $\mp1$ as follows.
If $m=2n$, then $I(1)=-I(1)$ and $I(-1)=-I(-1)$ give $I=(z^2-1)\widetilde I$ with $\widetilde I$ reciprocal of degree $2n-2$,
while $R$ is unconstrained reciprocal of degree $2n$.
If $m=2n+1$, then $R(-1)=-R(-1)$ gives $R=(z+1)\widetilde R$ and $I(1)=-I(1)$ gives $I=(z-1)\widetilde I$, with $\widetilde R,\widetilde I$ reciprocal of degree $2n$.
The factors $z^2-1$, $z+1$, $z-1$ are those whose squared moduli are the chord weights $s(4-s)$, $4-s$, $s$.
Conversely, for $\widetilde I$ reciprocal of degree $2n-2$ the product $(z^2-1)\widetilde I$ is anti-reciprocal of degree $2n$,
and for $\widetilde R,\widetilde I$ reciprocal of degree $2n$ the products $(z+1)\widetilde R$ and $(z-1)\widetilde I$
are reciprocal and anti-reciprocal of degree $2n+1$,
as follows from
\[\bigl((z^2-1)\widetilde I\bigr)^{*}=z^{2n}(z^{-2}-1)\widetilde I(1/z)=-(z^2-1)\widetilde I\]
and its two analogues.
Moreover, if $R_0$ is reciprocal and $I_0$ is anti-reciprocal of degree $m$ with leading coefficients summing to one,
then $\Phi=R_0+I_0$ is monic of degree $m$ with $\Phi^*=R_0-I_0$, so the splitting recovers $R=R_0$ and $I=I_0$.
The minimizations over the factored classes below therefore range over all potential minimands.

The real polynomials $G$ with $z^{2k}G(1/z)=G(z)$ form a space of dimension $k+1$,
with basis $z^{j}+z^{2k-j}$ for $0\le j\le k$, the element at $j=k$ being $2z^k$. On the circle, with $s=2-2\cos\theta$,
\[z^{-k}\bigl(z^{j}+z^{2k-j}\bigr)=2\cos\bigl((k-j)\theta\bigr)=2T_{k-j}\bigl(1-\tfrac s2\bigr),\qquad 0\le j\le k,\]
where $T_l$ is the Chebyshev polynomial of the first kind, normalized by $T_l(\cos\phi)=\cos(l\phi)$.
Since $2T_0=2$ and, for $l\ge 1$, $2T_l\bigl(1-\tfrac s2\bigr)$ has leading term $(-1)^l s^l$,
the images $2T_{l}\bigl(1-\tfrac s2\bigr)$, $0\le l\le k$, have degrees $0,1,\dots,k$
and form a basis of the real polynomials of degree at most $k$ in $s$.
The assignment $G\mapsto g$ defined by $z^{-k}G(z)=g(s)$ on the circle is therefore a linear bijection onto that space,
and the coefficient of $s^k$ in $g$ is $(-1)^k$ times the coefficient of $z^{2k}$ in $G$.
In particular, as $G$ ranges over the reciprocal polynomials with a prescribed coefficient of $z^{2k}$,
the image $g$ ranges over all real polynomials of degree at most $k$ with the corresponding coefficient of $s^k$.
This applies at $k=n$ to $R$ in the even case and to $\widetilde R$, $\widetilde I$ in the odd case,
and at $k=n-1$ to $\widetilde I$ in the even case.
Moreover $|z-1|^2=s$, $|z+1|^2=4-s$ and $|z^2-1|^2=s(4-s)$ on the circle.
Since the $s$-image of $\mu$ is $\omega$, each circle integral becomes a weighted half-line integral against $\omega$.
Each of the four minima $\fe_n$, $\fo_n$, $\fa_n$, $\fb_n$ is positive, because a nonzero polynomial vanishes on a finite set only,
while the measures $d\omega$, $s\,d\omega$, $(4-s)\,d\omega$ and $s(4-s)\,d\omega$ have infinite support,
the weights vanishing on $[0,4)$ only at $s=0$.

Suppose first that $m$ is even so that \(m=2n\) for $n\ge1$. Let \(\beta\) denote the coefficient of $z^{2n}$ in $R$.
Since $\Phi=R+I$ is monic, the coefficient of $z^{2n}$ in $I$ is $1-\beta$.
Moreover, since $I=(z^2-1)\widetilde I$ and $z^2-1$ is monic, $1-\beta$ is also the coefficient of
$z^{2n-2}$ in the reciprocal polynomial $\widetilde I$.

By the reciprocal-polynomial bijection established above, there are real polynomials $U$ and $V$ such that, on the unit circle,
\[z^{-n}R(z)=U(s) \quad\text{and}\quad z^{-(n-1)}\widetilde I(z)=V(s).\]
As $R$ varies with its $z^{2n}$-coefficient fixed at $\beta$, the polynomial $U$ varies over all real polynomials of degree at most $n$ whose \(s^n\)-coefficient is $(-1)^n\beta$.
Similarly, as $\widetilde I$ varies with its $z^{2n-2}$-coefficient fixed at $1-\beta$, the polynomial \(V\) varies over all real polynomials of
degree at most $n-1$ whose $s^{n-1}$-coefficient is $(-1)^{n-1}(1-\beta)$. The signs are fixed by the bijection and will not affect the minima, since the integrands are squares.

Because \(|z|=1\), we have
\[|R(z)|^2=U(s)^2.\]
Also,
\[|I(z)|^2 = |z^2-1|^2\,|\widetilde I(z)|^2 = s(4-s)V(s)^2.\]
Since the image of $\mu$ under $z\mapsto s=|z-1|^2$ is $\omega$,
it follows that
\[\int_{\partial\D}|R|^2\,d\mu = \int_{[0,4)} U^2\,d\omega \quad\text{and}\quad \int_{\partial\D}|I|^2\,d\mu = \int_{[0,4)} s(4-s)V^2\,d\omega.\]

Fix $\beta$. If $\beta\neq0$, every admissible $U$ can be written as $U=(-1)^n\beta Q$, where $Q$ is monic of degree $n$. Hence
\[\min_U\int_{[0,4)} U^2\,d\omega = \beta^2 \min_{\substack{Q\ \mathrm{monic}\\ \deg Q=n}}\int_{[0,4)} Q^2\,d\omega = \beta^2\fe_n.\]
When $\beta=0$, the admissible class contains $U=0$, so the minimum is $0$,
which is again $\beta^2\fe_n$. In the same way,
\[\min_V\int_{[0,4)} s(4-s)V^2\,d\omega = (1-\beta)^2\fo_n,\]
including the case $\beta=1$, when one may take $V=0$.

The reciprocal and anti-reciprocal parts may be minimized independently once $\beta$ is fixed.
Therefore
\[E_{2n}(\mu) = \min_{\beta\in\mathbb R}\bigl\{\beta^2\fe_n+(1-\beta)^2\fo_n\bigr\}.\]
Completing the square gives
\[\beta^2\fe_n+(1-\beta)^2\fo_n = (\fe_n+\fo_n)\left(\beta-\frac{\fo_n}{\fe_n+\fo_n}\right)^2 + \frac{\fe_n\fo_n}{\fe_n+\fo_n}.\]
Thus the minimum is attained at
\[\beta^\ast=\frac{\fo_n}{\fe_n+\fo_n},\]
and
\[E_{2n}(\mu)=\frac{\fe_n\fo_n}{\fe_n+\fo_n}.\]

It remains to determine the constant term of the extremal polynomial.
Since $R$ is reciprocal of degree $2n$, its constant term equals its $z^{2n}$-coefficient, and hence $R(0)=\beta$.
Likewise, $\widetilde I$ is reciprocal of degree $2n-2$, so $\widetilde I(0)=1-\beta$.
Since $I=(z^2-1)\widetilde I$, we have $I(0)=-\widetilde I(0)=-(1-\beta)$.
Consequently,
\[\Phi_{2n}(0) = R(0)+I(0) = \beta^\ast-(1-\beta^\ast) = \frac{\fo_n-\fe_n}{\fo_n+\fe_n}.\]

Now suppose that $m$ is odd so that $m=2n+1$ for $n\ge0$. In this case the factorizations are
\[R=(z+1)\widetilde R, \qquad I=(z-1)\widetilde I,\]
where $\widetilde R$ and $\widetilde I$ are reciprocal polynomials of degree $2n$.
Let $\beta$ be the coefficient of $z^{2n+1}$ in $R$.
Since $z+1$ is monic, $\beta$ is also the coefficient of $z^{2n}$ in $\widetilde R$.
The corresponding coefficients in $I$ and $\widetilde I$ are $1-\beta$.

Applying the reciprocal-polynomial bijection to $\widetilde R$ and $\widetilde I$, write
\[z^{-n}\widetilde R(z)=\widetilde U(s), \qquad z^{-n}\widetilde I(z)=\widetilde V(s).\]
Here $\widetilde U$ and $\widetilde V$ range over real polynomials of degree at most $n$, with $s^n$-coefficients
$(-1)^n\beta$ and $(-1)^n(1-\beta)$, respectively. Again using that the \(s\)-image of \(\mu\) is \(\omega\), we obtain
\[\int_{\partial\D}|R|^2\,d\mu = \int_{[0,4)} (4-s)\widetilde U(s)^2\,d\omega \quad\text{and}\quad
\int_{\partial\D}|I|^2\,d\mu = \int_{[0,4)} s\,\widetilde V(s)^2\,d\omega.\]

The same homogeneity argument as in the even case, including
$\beta=0$ and $\beta=1$, now gives
\[\min_{\widetilde U}\int (4-s)\widetilde U^2\,d\omega = \beta^2\fb_n \quad\text{and}\quad \min_{\widetilde V} \int s\widetilde V^2\,d\omega = (1-\beta)^2\fa_n.\]
It follows that
\[E_{2n+1}(\mu) = \min_{\beta\in\mathbb R}\bigl\{\beta^2\fb_n+(1-\beta)^2\fa_n\bigr\}.\]
Completing the square,
\[\beta^2\fb_n+(1-\beta)^2\fa_n = (\fa_n+\fb_n)\left(\beta-\frac{\fa_n}{\fa_n+\fb_n}\right)^2 + \frac{\fa_n\fb_n}{\fa_n+\fb_n}.\]
Hence the minimum is attained at
\[\beta^\ast=\frac{\fa_n}{\fa_n+\fb_n},\]
and
\[E_{2n+1}(\mu) = \frac{\fa_n\fb_n}{\fa_n+\fb_n}.\]

Finally, reciprocity gives
\[\widetilde R(0)=\beta,\qquad \widetilde I(0)=1-\beta.\]
Since
\[R=(z+1)\widetilde R,\qquad I=(z-1)\widetilde I,\]
we therefore have
\[R(0)=\widetilde R(0)=\beta, \qquad I(0)=-\widetilde I(0)=-(1-\beta).\]
Evaluating at the minimizing value $\beta^\ast$ yields
\[\Phi_{2n+1}(0) = \beta^\ast-(1-\beta^\ast) = \frac{\fa_n-\fb_n}{\fa_n+\fb_n}.\]
\end{proof}

Two of the four minima, $\fe_n$ and $\fa_n$, are respectively the plain norms $\|q_n\|_\omega^2$ and $\|r_n\|_\eta^2$ by the even/odd reduction;
we call these the \emph{plain} minima.
The other two, $\fo_n$ and $\fb_n$, carry the factor $4-s$;
we call them the \emph{tilted} minima and evaluate them using the Christoffel transform in the following lemma.

\begin{lemma}\label{lem:christoffel}
Let $c>0$ and let $\vartheta$ be a finite positive measure with infinite support contained in $[0,c)$ and monic orthogonal polynomials
$F_k$. Then for every $k\ge0$,
\[\min_{\substack{\deg Q=k\\ Q\ \mathrm{monic}}}\int_{[0,c)} (c-s)\,Q^2\,d\vartheta=\frac{F_{k+1}(c)}{F_k(c)}\,\|F_k\|_\vartheta^2.\]
\end{lemma}

\begin{proof}
Fix $k\geq 0$. Since $\vartheta$ has infinite support in $[0,c)$, we have that all of the zeros of $F_k$ are in $(0,c)$ \cite{Chihara},
and $F_k(c)>0$. Define $\hat F_k$ by
\[(c-s)\hat F_k(s)=\frac{F_{k+1}(c)}{F_k(c)}F_k(s)-F_{k+1}(s).\]
The right-hand side vanishes at $s=c$, so it is divisible by $c-s$.
Moreover, since $F_{k+1}$ is monic and $F_k$ has degree $k$, the right-hand side has leading term $-s^{k+1}$.
Since $c-s$ has leading term $-s$, it follows that $\hat F_k$ is monic of degree $k$. For any polynomial $T$ with $\deg T\le k-1$,
\[\int_{[0,c)} T\,\hat F_k\,(c-s)\,d\vartheta=\frac{F_{k+1}(c)}{F_k(c)}\int_{[0,c)} T F_k\,d\vartheta-\int_{[0,c)} T F_{k+1}\,d\vartheta=0,\]
because $F_k,F_{k+1}\perp T$ in $L^2(\vartheta)$.
Thus $\hat F_k$ is orthogonal to all lower degree polynomials in $L^2((c-s)d\vartheta)$, hence is the minimizer in that space. Its norm is
\[\int_{[0,c)} \hat F_k(c-s)\hat F_k\,d\vartheta = \int_{[0,c)}\hat F_k\Bigl[\tfrac{F_{k+1}(c)}{F_k(c)}F_k-F_{k+1}\Bigr]d\vartheta = \frac{F_{k+1}(c)}{F_k(c)}\|F_k\|^2,\]
since the $F_k$-coefficient of the monic degree-$k$ polynomial $\hat F_k$ is $1$ and $\int\hat F_kF_{k+1}d\vartheta=0$.
\end{proof}

One structural identity of the circle side remains, the factorization of the Toeplitz determinant
through the prediction errors.

\begin{lemma}\label{lem:gram}
Let $\mu$ be a finite positive measure on $\partial\D$ with infinite support,
let $\Phi_m$ be its monic orthogonal polynomials, and let $c_k=\int_{\partial\D}z^{-k}\,d\mu$. Then for every $N\ge0$,
\[D_N=\det\bigl[c_{j-k}\bigr]_{j,k=0}^{N}=\prod_{m=0}^{N}\|\Phi_m\|_\mu^2=\prod_{m=0}^{N}E_m(\mu).\]
\end{lemma}

\begin{proof}
The $(j,k)$ entry is $c_{j-k}=\int_{\partial\D}z^{k-j}\,d\mu=\langle z^k,z^j\rangle_\mu$,
so the matrix is the transpose of the Gram matrix of $1,z,\dots,z^N$ in $L^2(\mu)$ and thus has the same determinant.
The change of basis from $(z^m)_{m=0}^{N}$ to $(\Phi_m)_{m=0}^{N}$ is unit triangular,
each $\Phi_m$ being monic of degree $m$, so it preserves the Gram determinant,
and the Gram matrix of the $\Phi_m$ is diagonal with entries $\|\Phi_m\|_\mu^2$ by orthogonality.
Finally $E_m(\mu)=\|\Phi_m\|_\mu^2$, since the minimum of $\int|\Phi|^2\,d\mu$ over monic $\Phi$ of degree $m$
is the squared distance in $L^2(\mu)$ from $z^m$ to the polynomials of degree less than $m$,
attained at the projection residual, which is monic and orthogonal to all lower degrees, hence equals $\Phi_m$.
\end{proof}

With the plain minima identified, the tilted minima evaluated, and the determinant factored, we can now prove \Cref{thm:main}.

\begin{proof}[Proof of \Cref{thm:main}]
The four minima of \Cref{lem:split} reduce to data of $\sigma$ at the edge. By \Cref{eq:evenodd},
$\fe_n=\|q_n\|_\omega^2=\|P_{2n}\|_\sigma^2$ and $\fa_n=\|r_n\|_\eta^2=\|P_{2n+1}\|_\sigma^2$.
By \Cref{lem:christoffel} applied with $c=4$, once to $\vartheta=\eta$ at degree $n-1$ and once to $\vartheta=\omega$ at degree $n$, and using \Cref{eq:edgevalues},
\[\fo_n=\frac{r_n(4)}{r_{n-1}(4)}\,\|r_{n-1}\|_\eta^2 = \frac{P_{2n+1}(2)}{P_{2n-1}(2)}\,\|P_{2n-1}\|_\sigma^2,
\qquad \fb_n=\frac{q_{n+1}(4)}{q_n(4)}\,\|q_n\|_\omega^2 = \frac{P_{2n+2}(2)}{P_{2n}(2)}\,\|P_{2n}\|_\sigma^2 .\]
Define $\rho_m$ as the ratio of the two minima at level $m$, plain over tilted.
Then, with $t_m$ as in \Cref{eq:tdef} and $\|P_m\|_\sigma^2=A_m^2\|P_{m-1}\|_\sigma^2$ \cite{Chihara},
\[\rho_{2n}=\frac{\fe_n}{\fo_n} =\frac{\|P_{2n}\|_\sigma^2}{\|P_{2n-1}\|_\sigma^2}\cdot\frac{P_{2n-1}(2)}{P_{2n+1}(2)} = \frac{A_{2n}^2}{t_{2n}t_{2n-1}},
\quad \rho_{2n+1}=\frac{\fa_n}{\fb_n} =\frac{\|P_{2n+1}\|_\sigma^2}{\|P_{2n}\|_\sigma^2}\cdot\frac{P_{2n}(2)}{P_{2n+2}(2)} = \frac{A_{2n+1}^2}{t_{2n+1}t_{2n}},\]
so uniformly $\rho_m=A_m^2/(t_mt_{m-1})$ for $m\ge1$.
The crucial point here is that, by \Cref{eq:tdef}, $A_m^2/t_{m-1}=2-t_m$,
so the three edge quantities enter only through $t_m$, and $\rho_m=\frac{2-t_m}{t_m}$.
By \Cref{lem:split}, $E_m$ equals the numerator minimum divided by $1+\rho_m$, where the numerator minimum is $\fe_n=\|P_{2n}\|_\sigma^2$ for $m=2n$
and $\fa_n=\|P_{2n+1}\|_\sigma^2$ for $m=2n+1$. In both cases $E_m=\|P_m\|^2_\sigma\,t_m/2$, which is \Cref{eq:mainE} for the prediction errors,
the case $m=0$ being $E_0=\tau_0$ with $t_0=2$.
For the extremal values, by \Cref{lem:split} and $\rho_m=\frac{2-t_m}{t_m}$,
\[\Phi_{2n}(0)=\frac{1-\rho_{2n}}{1+\rho_{2n}}=t_{2n}-1, \qquad \Phi_{2n+1}(0)=\frac{\rho_{2n+1}-1}{\rho_{2n+1}+1}=1-t_{2n+1},\]
which is \Cref{eq:mainE} for $\Phi_m(0)$.
The margin formula follows either from $1-\Phi_m(0)^2=t_m(2-t_m)$ together with $(2-t_m)t_{m-1}=A_m^2$, or directly from the ratio of the norm formulas.
Finally \Cref{lem:gram} gives $D_N=\prod_{m\le N}E_m$, while $\prod_{m\le N}t_m=P_{N+1}(2)$ telescopes
from \Cref{eq:tdef}, giving \Cref{eq:maindet}.
\end{proof}

\begin{remark}\label{rem:secondkind}
In Simon's convention $\alpha_{m-1}=-\overline{\Phi_m(0)}$ \cite{Simon1}, so the Verblunsky coefficients of the lift are $\alpha_{m-1}=(-1)^{m+1}(t_m-1)$,
alternating in sign with magnitudes $t_m-1$ whenever $t_m>1$. The relations of \Cref{thm:main} are stated through $\Phi_m(0)$ to remain convention-free.
\end{remark}

\section{Monotonicity and Tur\'an-type inequalities}\label{sec:mono}

In this section we show that decay of the recurrence coefficients of $\sigma$ forces the $t_m$ to increase past $m=1$,
and that the increase is a Tur\'an-type inequality for $m\mapsto P_m(2)$.

\begin{lemma}\label{lem:tbounds}
For every $m\ge0$ we have $2-R\le t_m\le 2$, with $t_m<2$ for $m\ge1$.
Moreover $2-t_m\le A_m^2/(2-R)$ for every $m\ge1$, so $t_m\to2$ whenever $A_m^2\to0$.
\end{lemma}

\begin{proof}
The case $m=0$ is immediate from $t_0=2$. For $m\ge1$, since the zeros of $P_j$ are in $[-R,R]$,
we have $P_j(2)>0$ for every $j$, and hence $t_j>0$.
Therefore \Cref{eq:tdef} gives
\[t_m=2-\frac{A_m^2}{t_{m-1}}< 2,\]
where strict inequality holds since $A_m^2>0$.
For the lower bound, let $z_1^{(m)}\le\cdots\le z_m^{(m)}$ be the zeros of $P_m$, all in $[-R,R]$, interlacing those of $P_{m+1}$.
Writing $P_m(2)=\prod_i(2-z_i^{(m)})$, we have
\[t_m=\frac{\prod_{i=1}^{m+1}\bigl(2-z_i^{(m+1)}\bigr)}{\prod_{i=1}^{m}\bigl(2-z_i^{(m)}\bigr)}
=\bigl(2-z_{m+1}^{(m+1)}\bigr)\prod_{i=1}^m\frac{2-z_i^{(m+1)}}{2-z_i^{(m)}}\ \ge\ 2-R,\]
since interlacing gives $z_i^{(m+1)}\le z_i^{(m)}$ for $1\le i\le m$ and the last factor is at least
$2-R$. The final claim follows from $2-t_m=A_m^2/t_{m-1}\le A_m^2/(2-R)$.
\end{proof}

\begin{maintheorem}\label{thm:mono}
Suppose the recurrence coefficients of $\sigma$ are strictly decreasing, $A_1^2>A_2^2>\cdots$, and that
\begin{equation}\label{eq:basecond}
A_2^2\ <\ A_1^2\Bigl(1-\frac{A_1^2}{4}\Bigr).
\end{equation}
Then,
\begin{enumerate}
\item[(i)] For $m\geq 2$, $t_{m-1}<t_m$;
\item[(ii)] $m\mapsto P_m(2)$ is strictly log-convex from the second index on, that is
$P_{m+1}(2)\,P_{m-1}(2)>P_m(2)^2$ for $m\ge2$, while at the first index $t_1<t_0=2$ reverses this to $P_2(2)\,P_0(2)<P_1(2)^2$;
\item[(iii)] every prediction margin of the lift past the first strictly exceeds the corresponding recurrence coefficient,
\[1-\Phi_m(0)^2=A_m^2\,\frac{t_m}{t_{m-1}}>A_m^2,\qquad m\ge2,\]
with deviations $\log\bigl(t_m/t_{m-1}\bigr)$ positive and summable; if moreover $A_m^2\to0$, their total is
\[\sum_{m\ge2}\log\frac{t_m}{t_{m-1}} = \log\frac{4}{4-A_1^2}.\]
\end{enumerate}
\end{maintheorem}

Informally, \Cref{eq:basecond} requires the first drop, from $A_1^2$ to $A_2^2$,
to clear a relative correction of size $A_1^2/4$.
This is the base step of the induction behind (i),
and for the Lommel ensembles it will be verified with room to spare in \Cref{thm:lommel}(ii).

\begin{proof}
From \Cref{eq:tdef},
\[t_{m+1}-t_m=\frac{A_m^2}{t_{m-1}}-\frac{A_{m+1}^2}{t_m}=\frac{A_m^2t_m-A_{m+1}^2t_{m-1}}{t_{m-1}t_m},\] with $t_{m-1}t_m>0$ by \Cref{lem:tbounds}.
Now if $t_m\ge t_{m-1}$ then $A_m^2t_m\ge A_m^2t_{m-1}>A_{m+1}^2t_{m-1}$, so strict increase propagates by induction given the base step $t_2>t_1$,
which follows directly from \Cref{eq:basecond} since $A_2^2=2t_1-t_1t_2$ and $A_1^2=4-2t_1$ by \Cref{eq:tdef}.

Also from \Cref{eq:tdef}, the result in (ii) restates $t_m>t_{m-1}$ for $m\ge2$, with $t_1<t_0=2$ giving the reversed inequality $P_2(2)\,P_0(2)<P_1(2)^2$ at the first step.

The first part of (iii) follows from \Cref{eq:mainmargin}, with
\[\sum_{m=2}^N\log\frac{t_m}{t_{m-1}}=\log\frac{t_N}{t_1}\leq\log\frac{2}{t_1},\]
so the positive deviations are summable.
If $A_m^2\rightarrow 0$, then $t_N\rightarrow 2$ by \Cref{lem:tbounds}, and the asserted total follows from $t_1=2-A_1^2/2$.
\end{proof}

\section{Lommel ensembles}\label{sec:lommel}

We now choose the measure $\sigma$ so that its orthogonal polynomials are the Lommel family.
The Lommel polynomials are attached to the Bessel function of the first kind $J_\nu$. Their orthogonality measure is purely atomic, with atoms at the
reciprocals of the zeros of $J_\nu$, the recurrence coefficients rational in the degree.
Running these facts through \Cref{thm:main} makes each output explicit,
the norms in Gamma closed form, the edge limit a Bessel value, and the finite prediction quantities exact through the edge ratio $t_m$.
Throughout this section, implied constants in the $O$ notation may depend on $\nu$ and $\beta$.

For $\nu>-1$ and $0<\beta<2\sjz$, the \emph{Lommel ensemble} $\Lom_{\nu,\beta}$ is the purely atomic symmetric
measure placing mass $(\beta/j_{\nu,k})^2$ at each of the points $\pm\beta/j_{\nu,k}$, $k\ge1$;
its support is the closure $\{0\}\cup\{\pm\beta/j_{\nu,k}\colon k\ge1\}$, and the origin carries no mass.

The constraint on $\beta$ puts the support inside $(-2,2)$. The total mass is, by Rayleigh's identity
$\sum_k j_{\nu,k}^{-2}=\tfrac1{4(\nu+1)}$ \cite{Watson},
\begin{equation}\label{eq:tau0}
\tau_0=2\beta^2\sum_{k\ge1}j_{\nu,k}^{-2}=\frac{\beta^2}{2(\nu+1)} .
\end{equation}
The orthogonality of the modified Lommel polynomials for the discrete measure with masses
proportional to $j_{\nu,k}^{-2}$ at $\pm1/j_{\nu,k}$ goes back to Dickinson \cite{Dickinson} and Dickinson, Pollak and Wannier \cite{DickinsonPollakWannier}.
Maki \cite{Maki} treated the companion problem in the order variable, constructing the essentially unique measure
that orthogonalizes the family as polynomials in $\nu$ at fixed argument, with atoms at the $\nu$-zeros of the Bessel function.
Chihara \cite[Ch.~VI, \S 6]{Chihara} and Ismail's monograph \cite[\S 6.5]{Ismail} are the standard modern references. The family has been extended in three directions,
to the $q$-Lommel polynomials of the Hahn--Exton $q$-Bessel function \cite{Koelink},
to orthogonal polynomials whose orthogonality measure has atoms at reciprocals of zeros of Coulomb wave functions \cite{StampachStovicek},
and to orthogonal polynomials whose orthogonality measure has atoms at reciprocals of zeros of the derivative $J_\nu'$ \cite{ChungLeePark}.
The lift applies to any of these families whose measure is symmetric with compact support inside $(-2,2)$,
since \Cref{thm:main} asks nothing more;
an analogue of \Cref{thm:lommel} would further require closed recurrence data and a Hurwitz-type edge limit,
but we have not pursued these calculations.

The masses $(\beta/j_{\nu,k})^2$ are proportional to $j_{\nu,k}^{-2}$,
so $\Lom_{\nu,\beta}$ is a constant multiple of the $\beta$-dilation of Dickinson's measure
and its monic orthogonal polynomials are the dilated modified Lommel polynomials.
The masses are also the squares of the atom locations, and $x^2=|1-z|^2$ along the lift,
so each atom of $\Lambda\Lom_{\nu,\beta}$ carries the Christoffel weight $|1-z|^2$ evaluated there.
It is tempting to conclude that the lift is the Christoffel transform at $z=1$ of some finite measure,
which would deliver the Verblunsky coefficients of the lift from those of the base measure in one step. It is not.
Each atom contributes $1$ to $\int_{\partial\D}|1-z|^{-2}\,d\bigl(\Lambda\Lom_{\nu,\beta}\bigr)$,
so the integral diverges and no finite base measure exists.

Normalizations differ across the Lommel literature, so we state our conventions in full,
and we prove the identification of the orthogonality measure on the full range $\nu>-1$ in \Cref{prop:lommelid}.
The precise normalization is that of \cite{DickinsonPollakWannier}, whose modified Lommel polynomials $L_n^{(s)}$ are monic and satisfy
\begin{equation}\label{eq:dpwrec}
L_0^{(s)}=1,\qquad L_1^{(s)}(x)=x,\qquad
L_{n+1}^{(s)}(x)=x\,L_n^{(s)}(x)-\frac{1}{4(s+n)(s+n+1)}\,L_{n-1}^{(s)}(x),\quad n\ge1,
\end{equation}
with $L_n^{(s)}(x)=R_{n,s+1}(1/x)\,\Gamma(s+1)/\bigl[2^n\Gamma(s+n+1)\bigr]$ in terms of Watson's Lommel polynomials $R_{n,\nu}$ \cite{Watson}.
The coefficient in \Cref{eq:dpwrec} is positive for every $n\ge1$ precisely when $s>-1$.
Dickinson, Pollak and Wannier \cite{DickinsonPollakWannier} display the recurrence for $n\ge0$ with $L_{-1}^{(s)}=0$,
and the $n=0$ instance forces their standing restriction $s>0$;
started at $n\ge1$ as above, the family is defined, with positive and summable recurrence coefficients,
for every $s>-1$, and agrees with theirs when $s>0$.
Their orthogonality relation, equation (25) of \cite{DickinsonPollakWannier}, stated for $s>0$,
places masses proportional to $j_{s,k}^{-2}$ at the points $\pm1/j_{s,k}$;
\Cref{prop:lommelid} recovers it, up to the overall normalization, and extends it to every $s>-1$.

Take $s=\nu$, write $a_m^2=1/[4(\nu+m)(\nu+m+1)]$ for the coefficients in \Cref{eq:dpwrec},
and pass to scale $\beta$ through the monic dilation
\begin{equation}\label{eq:monicdilation}
P_m(x)=\beta^m\,L_m^{(\nu)}(x/\beta),\qquad m\ge0.
\end{equation}
Substituting $x/\beta$ for the argument in \Cref{eq:dpwrec} and multiplying by $\beta^{m+1}$ gives
\[P_{m+1}(x)=x\,P_m(x)-\beta^2a_m^2\,P_{m-1}(x)\]
for $m\ge1$, so the support dilation $y\mapsto\beta y$ multiplies each recurrence coefficient by $\beta^2$.
Hurwitz's theorem \cite{Hurwitz}, in Watson's formulation \cite[\S 9.65\,(1)]{Watson}
with the branch factors divided out, states that for every real $s$,
\begin{equation}\label{eq:watsonhurwitz}
\lim_{m\to\infty}\frac{(z/2)^{m}\,R_{m,s+1}(z)}{\Gamma(s+m+1)}=\Bigl(\frac z2\Bigr)^{-s}J_s(z),
\end{equation}
uniformly on bounded subsets of the $z$-plane;
both sides are even entire functions of $z$, the right side through its power series, so no branch convention enters.
The foregoing discussion establishes, for every $m\ge0$ and $x\ne0$, the identity
\begin{equation}\label{eq:monictransport}
\frac{P_m(x)}{x^m}=\Gamma(\nu+1)\,\frac{(z/2)^{m}\,R_{m,\nu+1}(z)}{\Gamma(\nu+m+1)},\qquad z=\frac{\beta}{x},
\end{equation}
so \Cref{eq:watsonhurwitz} at $s=\nu$, applied to \Cref{eq:monictransport} together with the canonical product
\[J_\nu(z)=\frac{(z/2)^\nu}{\Gamma(\nu+1)}\prod_{k\ge1}\bigl(1-z^2/j_{\nu,k}^2\bigr)\]
\cite[\S 15.41]{Watson}, yields
\begin{equation}\label{eq:hurwitz}
\lim_{m\to\infty}\frac{P_m(x)}{x^m} = \prod_{k\ge1}\Bigl(1-\frac{\beta^2}{j_{\nu,k}^2x^2}\Bigr) = \Gamma(\nu+1)\Bigl(\frac{2x}{\beta}\Bigr)^{\nu} J_\nu\Bigl(\frac{\beta}{x}\Bigr),
\end{equation}
the first equality locally uniformly on compact subsets of $\{x\in\C\colon x\ne0\}$;
the powers of $x$ cancel between the two factors of the third member,
so that member is branch independent and equals the product for complex $x$ as well.

\begin{proposition}\label{prop:lommelid}
Let $\nu>-1$ and $0<\beta<2\sjz$.
The family \Cref{eq:monicdilation} is orthogonal with respect to $\Lom_{\nu,\beta}$,
and $\Lom_{\nu,\beta}$ is the unique positive measure with total mass $\tau_0$ having this property.
In particular the monic orthogonal polynomials of $\Lom_{\nu,\beta}$ satisfy the three-term recurrence with coefficients
\begin{equation}\label{eq:lommelA2}
A_m^2\bigl(\Lom_{\nu,\beta}\bigr)=\frac{\beta^2}{4(\nu+m)(\nu+m+1)},\qquad m\ge1,
\end{equation}
strictly decreasing in $m$, with norms
\begin{equation}\label{eq:lommelnorms}
\|P_m\|^2_{\Lom_{\nu,\beta}} = \tau_0\prod_{j=1}^m A_j^2 = \frac{\beta^2}{2(\nu+1)}\,\Bigl(\frac{\beta}{2}\Bigr)^{2m}\,
\frac{\Gamma(\nu+1)\,\Gamma(\nu+2)}{\Gamma(\nu+m+1)\,\Gamma(\nu+m+2)} .
\end{equation}
The support of $\Lom_{\nu,\beta}$ is $\{0\}\cup\{\pm\beta/j_{\nu,k}\colon k\ge1\}$, and the origin carries no mass.
\end{proposition}

\begin{proof}
The coefficients $\beta^2a_m^2$ are positive for $m\ge1$ and bounded,
so by Favard's theorem the family \Cref{eq:monicdilation} is orthogonal with respect to a positive measure $\mu$
with total mass $\tau_0$, and the boundedness of the coefficients places the support of $\mu$ in a compact set
and makes $\mu$ the unique measure with its moments \cite[Chs.~II and IV]{Chihara}.
The second kind polynomials $Q_m(x)=\int\bigl(P_m(x)-P_m(t)\bigr)(x-t)^{-1}\,d\mu(t)$
satisfy the same recurrence with $Q_0=0$ and $Q_1=\tau_0$,
and $a_{m+1}^2$ is the coefficient at index $m$ in the parameter $\nu+1$ recurrence,
so induction gives $Q_{m+1}(x)=\tau_0\,\beta^{m}L_m^{(\nu+1)}(x/\beta)$ for $m\ge0$.
By Markov's theorem \cite[\S 2.6]{Ismail}, for $x\in\C\setminus\R$,
\[\int_\R\frac{d\mu(t)}{x-t}=\lim_{m\to\infty}\frac{Q_{m+1}(x)}{P_{m+1}(x)}
=\frac{\tau_0}{x}\,\frac{\prod_{k\ge1}\bigl(1-\beta^2/(j_{\nu+1,k}^2x^2)\bigr)}{\prod_{k\ge1}\bigl(1-\beta^2/(j_{\nu,k}^2x^2)\bigr)}
=\beta\,\frac{J_{\nu+1}(\beta/x)}{J_\nu(\beta/x)},\]
the second equality by \Cref{eq:hurwitz} at parameters $\nu+1$ and $\nu$,
whose denominator limit is zero free off the real axis,
and the third by the canonical products of $J_{\nu+1}$ and $J_\nu$ together with $\tau_0=\beta^2/(2(\nu+1))$.
Logarithmic differentiation of the canonical product,
with the recurrence $J_\nu'(w)=\tfrac{\nu}{w}J_\nu(w)-J_{\nu+1}(w)$, gives
\[\frac{J_{\nu+1}(w)}{J_\nu(w)}=\sum_{k\ge1}\frac{2w}{j_{\nu,k}^2-w^2},\]
locally uniformly off the zeros, and the substitution $w=\beta/x$ turns the Stieltjes transform of $\mu$ into
\[\int_\R\frac{d\mu(t)}{x-t} = \sum_{k\ge1}\Bigl(\frac{\beta}{j_{\nu,k}}\Bigr)^{2}
\Bigl(\frac{1}{x-\beta/j_{\nu,k}}+\frac{1}{x+\beta/j_{\nu,k}}\Bigr) = \int_\R\frac{d\Lom_{\nu,\beta}(t)}{x-t}.\]
Two finite positive measures with the same Stieltjes transform on $\C\setminus\R$ coincide, so $\mu=\Lom_{\nu,\beta}$.
The support of $\Lom_{\nu,\beta}$ is the closure of its atom set, which adjoins the accumulation point $0$.
The norm formula \Cref{eq:lommelnorms} is the identity $\|P_m\|_{\Lom_{\nu,\beta}}^2=\tau_0\prod_{j\le m}A_j^2$ \cite{Chihara} together with
$\prod_{j=1}^m(\nu+j)^{-1}=\Gamma(\nu+1)/\Gamma(\nu+m+1)$ and $\prod_{j=1}^m(\nu+j+1)^{-1}=\Gamma(\nu+2)/\Gamma(\nu+m+2)$.
\end{proof}

In particular the orthogonality measure of the recurrence places no mass at the accumulation point $0$ of its support,
the point left open in \cite{DickinsonPollakWannier} and settled by Goldberg \cite{Goldberg};
the Rayleigh identity \Cref{eq:tau0} confirms that the atoms alone exhaust the total mass.

\begin{maintheorem}\label{thm:lommel}
Let $\mu_{\nu,\beta}=\Lambda\Lom_{\nu,\beta}$ and let $t_m$ be defined by
\begin{equation}\label{eq:lommeltrec}
t_0=2,\qquad t_m=2-\frac{\beta^2}{4(\nu+m)(\nu+m+1)\,t_{m-1}},\quad m\ge1.
\end{equation}
Then the prediction errors, extremal values and Toeplitz determinants of $\mu_{\nu,\beta}$ are given by \Cref{thm:main} applied with $\sigma=\Lom_{\nu,\beta}$,
with the data \Cref{eq:lommelA2}, \Cref{eq:lommelnorms}; in
particular
\[E_m(\mu_{\nu,\beta}) = \frac{t_m}{2}\cdot\frac{\beta^2}{2(\nu+1)}\Bigl(\frac{\beta}{2}\Bigr)^{2m}
\frac{\Gamma(\nu+1)\Gamma(\nu+2)}{\Gamma(\nu+m+1)\Gamma(\nu+m+2)}, \qquad \Phi_m(0)=(-1)^m\,(t_m-1).\]
Moreover,
\begin{enumerate}
\item[(i)] $\displaystyle \prod_{m=1}^{\infty}\frac{t_m}{2}
=\Gamma(\nu+1)\Bigl(\frac{4}{\beta}\Bigr)^{\nu}J_\nu\Bigl(\frac{\beta}{2}\Bigr) =:\mathcal J_{\nu,\beta}\in(0,1)$,
and the Toeplitz determinants satisfy the identity and normalization limit
\[D_N=\frac{P_{N+1}(2)}{2^{N+1}}\prod_{m=0}^N\|P_m\|_{\Lom_{\nu,\beta}}^2, \qquad \frac{P_{N+1}(2)}{2^{N+1}}\ \longrightarrow\ \mathcal J_{\nu,\beta}\quad\text{as}\quad N\to\infty.\]
\item[(ii)] If
\begin{equation}\label{eq:lommelbase}
\beta^2<\frac{32(\nu+1)(\nu+2)}{\nu+3},
\end{equation}
then $t_1<t_2<\cdots$ and all conclusions of \Cref{thm:mono} hold. In particular every margin past the first strictly exceeds $A_m^2$.
\item[(iii)] As $m\to\infty$, the margin deviations obey the third-order law
\[\log\frac{1-\Phi_m(0)^2}{A_m^2}=\log\frac{t_m}{t_{m-1}} = \frac{\beta^2}{8(\nu+m-1)(\nu+m)(\nu+m+1)}\,\bigl(1+O(m^{-2})\bigr),\]
and $2-t_m=\tfrac12A_m^2\bigl(1+O(m^{-2})\bigr)$; since $t_m\to2$, for all sufficiently large $m$ we have $t_m>1$
and $1-|\Phi_m(0)|=2-t_m$, so $1-|\Phi_m(0)|=\dfrac{\beta^2}{8(\nu+m)(\nu+m+1)}\bigl(1+O(m^{-2})\bigr)$.
\end{enumerate}
\end{maintheorem}

\begin{proof}
\Cref{prop:lommelid} supplies the data \Cref{eq:lommelA2} and \Cref{eq:lommelnorms}, and the explicit forms are then \Cref{thm:main}.
For (i), telescoping gives $\prod_{m=0}^{N}t_m=P_{N+1}(2)$, so $\prod_{m=1}^N(t_m/2)=P_{N+1}(2)/2^{N+1}$,
which converges to the right side of \Cref{eq:hurwitz} at $x=2$;
the limit lies in $(0,1)$, since $\beta/2<\sjz$ places every factor of the canonical product in $(0,1)$
while the summability of $j_{\nu,k}^{-2}$ in \Cref{eq:tau0} keeps the product positive. The determinant identity is \Cref{eq:maindet}.

For (ii), the coefficients \Cref{eq:lommelA2} are strictly decreasing, and condition \Cref{eq:basecond} reads $\frac{\nu+1}{\nu+3}< 1-\frac{\beta^2}{16(\nu+1)(\nu+2)}$,
which rearranges to \Cref{eq:lommelbase}; \Cref{thm:mono} then applies.

For (iii), write $\delta_m=2-t_m$, so that \Cref{eq:lommeltrec} gives $\delta_m=A_m^2/(2-\delta_{m-1})$.
From $(2-\delta)^{-1}=\tfrac12\bigl(1+\tfrac{\delta}{2}+\tfrac{\delta^2}{2(2-\delta)}\bigr)$ this recursion is the identity
\[\delta_m=\frac{A_m^2}{2}\Bigl(1+\frac{\delta_{m-1}}{2}+\frac{\delta_{m-1}^2}{2\,t_{m-1}}\Bigr).\]
By \Cref{lem:tbounds}, $t_{m-1}\ge2-R$ and $\delta_{m-1}\le A_{m-1}^2/(2-R)$,
and $A_m^2=O(m^{-2})$ by \Cref{eq:lommelA2}, so the last term inside the bracket is $O(m^{-4})$ and
$\delta_m=\tfrac12A_m^2\bigl(1+O(m^{-2})\bigr)$.
Taking this at $m-1$ gives $\delta_{m-1}=\tfrac12A_{m-1}^2+O(m^{-4})$; substituting into the identity,
both error sources, the substitution error and the $\delta_{m-1}^2$ term, carry the prefactor
$\tfrac12A_m^2=O(m^{-2})$ and so contribute $O(m^{-6})$, giving
\[\delta_m=\tfrac12A_m^2+\tfrac18A_m^2A_{m-1}^2+O(m^{-6}).\]
While the iteration can be continued to higher order, the two terms displayed suffice here.
Differencing through $t_m-t_{m-1}=\delta_{m-1}-\delta_m$, the quartic terms contribute
$\tfrac18A_{m-1}^2\bigl(A_{m-2}^2-A_m^2\bigr)$; by \Cref{eq:lommelA2}, $A_{m-1}^2=O(m^{-2})$ and
\[A_{m-2}^2-A_m^2=\frac{\beta^2\bigl(2(\nu+m)-1\bigr)}{2(\nu+m-2)(\nu+m-1)(\nu+m)(\nu+m+1)}=O(m^{-3}),\]
so this contribution is $O(m^{-5})$ and
\[t_m-t_{m-1}=\frac{A_{m-1}^2-A_m^2}{2}+O(m^{-5}), \qquad A_{m-1}^2-A_m^2=\frac{\beta^2}{2(\nu+m-1)(\nu+m)(\nu+m+1)}.\]
With $x_m=(t_m-t_{m-1})/t_{m-1}=O(m^{-3})$ and $1/t_{m-1}=\tfrac12\bigl(1+O(m^{-2})\bigr)$ from $\delta_{m-1}=O(m^{-2})$,
\[\log(t_m/t_{m-1})=x_m\bigl(1+O(x_m)\bigr)=\tfrac12(t_m-t_{m-1})\bigl(1+O(m^{-2})\bigr),\]
which gives the first asymptotic. The estimate just obtained gives
\[\delta_m=2-t_m=\tfrac12A_m^2\bigl(1+O(m^{-2})\bigr) = \frac{\beta^2}{8(\nu+m)(\nu+m+1)}\bigl(1+O(m^{-2})\bigr).\]
Since $t_m\to2$, we have $t_m>1$ for all sufficiently large $m$. Hence, using $\Phi_m(0)=(-1)^m(t_m-1)$, we obtain for all sufficiently large $m$,
\[1-|\Phi_m(0)|=2-t_m=\delta_m,\]
and the second asserted asymptotic follows.
\end{proof}

\begin{example}\label{ex:halfinteger}
The fully elementary case is $\nu=\tfrac12$, where the Bessel zeros are $j_{1/2,k}=\pi k$ and the canonical product is $\mathcal J_{1/2,\beta}=\frac{\sin(\beta/2)}{\beta/2}$.
With $\beta=\pi/2$ the ensemble is the purely atomic measure with mass $(2k)^{-2}$ at each of the points $\pm1/(2k)$, $k\ge1$,
the reciprocals of the nonzero even integers.
The deviation constant in (iii) is $\beta^2/8=\pi^2/32$, the limit is $\mathcal J=\operatorname{sinc}(\pi/4)=2\sqrt2/\pi$,
and condition \Cref{eq:lommelbase} holds with a wide margin, $\pi^2/4< 240/7$.
\end{example}

\begin{remark}\label{rem:novelty}
Explicit hypergeometric families of orthogonal polynomials on the unit circle were developed by Sri Ranga and collaborators \cite{SriRanga2010,MFSRT};
in the complementary Romanovski--Routh setting the measures are absolutely continuous of Fisher--Hartwig/circular-Jacobi type,
with Bessel and Coulomb wave functions entering through the generating functions.
The Lommel lifts sit at the opposite end of the correspondence.
The measures are purely atomic, the Bessel function enters through the support and the Hurwitz limit of the edge ratios,
and the prediction errors and Toeplitz determinants are given explicitly through the edge data.
\end{remark}

\section{The Christoffel function evaluated at the edge}\label{sec:christoffel}
\Cref{thm:main} builds each output from two real-line sequences, the edge values $P_m(2)$ and the norms $\|P_m\|_\sigma^2$.
The Christoffel function of $\sigma$, from classical orthogonal-polynomial theory, combines those two sequences at the edge by summation over the degree.
Reading the edge data through it recovers the determinant formula and Bessel limit as Christoffel-function statements, fixes what the constant
$\mathcal J_{\nu,\beta}$ measures, and gives closed-form asymptotics for the Christoffel function of the Lommel ensemble.
The variational formalism is the real-line one of \cite{SimonDescendants}, \S1.2.

For a measure $\sigma$ with monic orthogonal polynomials $P_j$ and orthonormal polynomials
$p_j=P_j/\|P_j\|_\sigma$, the \emph{reproducing kernel} and \emph{Christoffel function} at a real
point $\zeta$ are
\begin{equation}\label{eq:cdkernel}
\begin{aligned}
K_n(x,y;\sigma)&=\sum_{j=0}^n p_j(x)\,p_j(y),\\
\text{and} \quad \lambda_n(\zeta;\sigma)&=\frac{1}{K_n(\zeta,\zeta;\sigma)}
=\min\Bigl\{\int|Q|^2\,d\sigma:\ \deg Q\le n,\ Q(\zeta)=1\Bigr\}.
\end{aligned}
\end{equation}
The kernel reproduces polynomials of degree at most $n$ against $\sigma$. The minimum is attained at the normalized kernel $K_n(\,\cdot\,,\zeta)/K_n(\zeta,\zeta)$.

Specializing \Cref{eq:cdkernel} at $\zeta=2$ and writing $u_j=P_j(2)^2/\|P_j\|_\sigma^2$ for the
squared orthonormal edge values,
\begin{equation}\label{eq:Kedge}
K_n(2,2;\sigma)=\sum_{j=0}^n \frac{P_j(2)^2}{\|P_j\|_\sigma^2}=\sum_{j=0}^n u_j,\qquad
\lambda_n(2;\sigma)=\Bigl(\sum_{j=0}^n u_j\Bigr)^{-1}.
\end{equation}
The summands follow a one-step recursion in the edge sequence. From
\[P_j(2)/P_{j-1}(2)=t_{j-1},\quad\|P_j\|_\sigma^2/\|P_{j-1}\|_\sigma^2=A_j^2,\]
and $A_j^2=(2-t_j)\,t_{j-1}$ behind \Cref{eq:tdef},
\begin{equation}\label{eq:uratio}
\frac{u_j}{u_{j-1}}=\frac{t_{j-1}^2}{A_j^2}=\frac{t_{j-1}}{2-t_j}.
\end{equation}

\begin{lemma}\label{lem:edgechristoffel}
Suppose $A_m^2\to0$, so that $t_m\to2$ by \Cref{lem:tbounds}. Then the sum \Cref{eq:Kedge} is dominated
by its last term, and
\[\lambda_n(2;\sigma)=\frac{\|P_n\|_\sigma^2}{P_n(2)^2}\Bigl(1-\frac{2-t_n}{t_{n-1}}+o(2-t_n)\Bigr),\qquad\text{so}\qquad
\lambda_n(2;\sigma)\sim\frac{\|P_n\|_\sigma^2}{P_n(2)^2}.\]
In terms of the prediction errors, $\lambda_n(2;\sigma)\sim E_n/P_n(2)^2$.
\end{lemma}

Informally, the lemma asserts that the edge kernel is already its last summand,
up to a relative correction of the size of the chain gap $2-t_n$.

\begin{proof}
Set $r_m=(2-t_m)/t_{m-1}$ for $m\ge1$. By \Cref{eq:uratio} the ratio of consecutive summands is
$u_{m-1}/u_m=r_m$, by \Cref{eq:tdef} $r_m=A_m^2/t_{m-1}^2$, and by \Cref{lem:tbounds} $t_m\ge2-R$,
so $r_m\le A_m^2/(2-R)^2\to0$ and $\bar r=\sup_m r_m<\infty$. Factoring out the last term,
\[K_n(2,2;\sigma)=u_n\Bigl(1+\sum_{i=1}^{n}\ \prod_{l=0}^{i-1}r_{n-l}\Bigr).\]
Fix $\rho\in(0,1)$ and $N_0$ with $r_m\le\rho$ for $m>N_0$.
A descending product $\prod_{l=0}^{p-1}r_{j-l}$ of $p$ consecutive ratios contains at most $N_0$ factors
with index at most $N_0$, each at most $\max(\bar r,1)$, and every other factor is at most $\rho$,
so the product is at most $M\rho^{\,p}$ with $M=\bigl(\max(\bar r,1)/\rho\bigr)^{N_0}$,
uniformly in the top index $j$ and in $p$.
Writing $S_j=\sum_{p=0}^{j}\prod_{l=0}^{p-1}r_{j-l}$ for the finite sum of such products down to index $1$,
it follows that $S_j\le M/(1-\rho)$ for every $j$.
For $n\ge2$ the bracket regroups as $1+r_n+r_nr_{n-1}S_{n-2}$, and
\[r_nr_{n-1}S_{n-2}\ \le\ \frac{M}{1-\rho}\cdot\frac{r_n\,(2-t_{n-1})}{2-R}\ =\ o(2-t_n),\]
since $r_n\le(2-t_n)/(2-R)$ and $2-t_{n-1}\to0$. Hence the bracket is $1+r_n+o(2-t_n)$.
Since $r_n=O(2-t_n)$, inverting gives $\lambda_n(2;\sigma)=u_n^{-1}\bigl(1-r_n+o(2-t_n)\bigr)$,
the quadratic term being $O\bigl((2-t_n)^2\bigr)=o(2-t_n)$,
which is the first display, and its leading term is the asymptotic.
For the last statement, $E_n=(t_n/2)\|P_n\|_\sigma^2$ from \Cref{eq:mainE} and $t_n\to2$ give
$\|P_n\|_\sigma^2\sim E_n$ up to the factor $2/t_n\to1$.
\end{proof}

\begin{remark}
At a point of the support the Christoffel function converges to the mass there, $\lambda_n(x_\ast;\sigma)\to\sigma(\{x_\ast\})$ \cite{SimonDescendants}.
The edge $x=2$ lies strictly outside $[-R,R]$, so $\lambda_n(2;\sigma)\to0$ instead,
the reciprocal $K_n(2,2)=\sum u_j$ growing because the ratios $u_n/u_{n-1}=t_{n-1}/(2-t_n)$ diverge as $2-t_n\to0$, at the rate stated in the lemma.
\end{remark}

The kernel also has a closed Christoffel-Darboux form, useful when the polynomials are explicit.
In the monic normalization the confluent identity reads
\begin{equation}\label{eq:confluentCD}
K_n(x,x;\sigma)=\frac{P_{n+1}'(x)\,P_n(x)-P_n'(x)\,P_{n+1}(x)}{\|P_n\|_\sigma^2},
\end{equation}
so $\lambda_n(2;\sigma)=\|P_n\|_\sigma^2/W_n(2)$ with $W_n=P_{n+1}'P_n-P_n'P_{n+1}$ a polynomial
of degree $2n$ in the variable $x$ whose coefficients are determined by the recurrence data, evaluated once at the edge. For families with explicit polynomials
\Cref{eq:confluentCD} is an exact finite formula, and \Cref{lem:edgechristoffel} is its large-$n$ form.

By summation over the degree the edge data determine the Christoffel function. If multiplied instead, the norms form a determinant.
Write $h_{N+1}(\sigma)=\det\bigl[\int x^{i+j}\,d\sigma\bigr]_{i,j=0}^{N}=\prod_{m=0}^N\|P_m\|_\sigma^2$
for the Hankel determinant of $\sigma$, the line-side object behind Simon's identities
$\|P_n\|_\sigma^2=h_{n+1}/h_n$ and $A_n^2=h_{n-1}h_{n+1}/h_n^2$. The determinant formula of \Cref{thm:main}
then factors the Toeplitz determinant of the circle measure as the Hankel determinant of the line measure times one edge value,
\begin{equation}\label{eq:dethankel}
D_N(\mu)=\frac{P_{N+1}(2)}{2^{N+1}}\,h_{N+1}(\sigma).
\end{equation}
For the Lommel family the ratio of the two determinants converges,
\[\frac{D_N(\mu_{\nu,\beta})}{h_{N+1}(\Lom_{\nu,\beta})}=\frac{P_{N+1}(2)}{2^{N+1}}\xrightarrow[N\to\infty]{}\mathcal J_{\nu,\beta},\]
by \Cref{thm:lommel}(i). This fixes the meaning of $\mathcal J_{\nu,\beta}$ as the limiting ratio of the circle Toeplitz determinant to the line Hankel determinant,
the edge value $P_{N+1}(2)/2^{N+1}$ being all that separates them.
The same constant sets the scale of the edge Christoffel function through $P_n(2)\sim\mathcal J_{\nu,\beta}\,2^n$,
where \Cref{lem:edgechristoffel} meets \Cref{eq:dethankel}.

\begin{corollary}\label{cor:lommelchristoffel}
For the Lommel ensemble $\Lom_{\nu,\beta}$, with $\mathcal J_{\nu,\beta}$ as in \Cref{thm:lommel}(i),
\[\lambda_n(2;\Lom_{\nu,\beta})\sim\frac{\|P_n\|_{\Lom_{\nu,\beta}}^2}{\mathcal J_{\nu,\beta}^2\,4^n} = \frac{\beta^2}{2(\nu+1)\,\mathcal J_{\nu,\beta}^2}\Bigl(\frac{\beta}{4}\Bigr)^{2n}
\frac{\Gamma(\nu+1)\Gamma(\nu+2)}{\Gamma(\nu+n+1)\Gamma(\nu+n+2)}.\]
The decay is faster than any geometric rate, with $\lambda_n/\lambda_{n-1}\sim A_n^2/4 = (\beta/4)^2/\bigl((\nu+n)(\nu+n+1)\bigr)\to0$.
\end{corollary}

\begin{proof}
\Cref{lem:edgechristoffel} gives $\lambda_n\sim\|P_n\|_{\Lom_{\nu,\beta}}^2/P_n(2)^2$,
and \Cref{thm:lommel}(i) gives $P_n(2)/2^n\to\mathcal J_{\nu,\beta}$, hence $P_n(2)^2\sim\mathcal J_{\nu,\beta}^2\,4^n$.
The closed form for $\|P_n\|_{\Lom_{\nu,\beta}}^2$ is \Cref{eq:lommelnorms}, and $(\beta/2)^{2n}/4^n=(\beta/4)^{2n}$.
The ratio statement is \Cref{eq:uratio} with $2-t_n=\tfrac12 A_n^2(1+O(n^{-2}))$ from \Cref{thm:lommel}(iii).
\end{proof}

\section*{Disclosure of AI Tool Use}

This research was human-directed and carried out with assistance from AI tools.
The systems used were Claude Opus 4.8 and Claude Fable 5 (Anthropic), and GPT 5.5 and GPT 5.6 Sol (OpenAI).
Prompted by the author, Claude orchestrated high-precision numerical experiments in Python using \texttt{mpmath},
and the identification central to \Cref{thm:lommel} emerged serendipitously
in the course of a question the author had posed for other reasons.
Claude was used interactively to develop an initial sketch of the paper,
which the author used as a template for writing the final manuscript.
Every tool-suggested passage, argument, and reference appearing in the paper was verified and edited by the author,
and all tool-suggested proofs were independently checked and rewritten.
ChatGPT was used for proofreading.
The author takes full responsibility for the content of this work.

\end{document}